\documentclass[12pt]{article}

\usepackage[all]{xypic}
\usepackage{amsmath}
\usepackage{amsfonts}
\usepackage{amssymb}
\usepackage{amsthm}
\usepackage{lscape}
\usepackage{longtable}
\usepackage{float}
\usepackage{graphicx}

\usepackage[top=2cm, bottom=2cm, left=2cm, right=2cm]{geometry}

\theoremstyle{definition}
\newtheorem{thm}{Theorem}[section]

\theoremstyle{definition}
\newtheorem{crl}[thm]{Corollary}

\theoremstyle{definition}

\theoremstyle{definition}

\newtheorem{lmm}[thm]{Lemma}

\theoremstyle{definition}

\theoremstyle{definition}

\theoremstyle{definition}

\newcommand {\mf}{\mathbf}
\newcommand {\mr}{\mathrm}
\newcommand {\tb}{\textbf}
\newcommand {\mb}{\mathbb}
\newcommand {\Z}{\mb Z}
\newcommand {\R}{\mb R}
\newcommand {\C}{\mb C}
\newcommand {\Q}{\mb H}
\newcommand {\F}{\mb F}
\newcommand {\N}{\mb N}
\newcommand {\s}{\mathcal{S}}
\newcommand {\om}{\mathcal{O}}
\newcommand {\im}{\textrm{im}}
\newcommand {\colim}{\textrm{colim}\ }
\newcommand {\cok}{\textrm{coker}}
\newcommand {\tr}{\textrm{tr}}
\newcommand {\nto}{\not\to}
\newcommand {\ex}{\mathrm{ex}}
\newcommand {\ext}{\mathrm{Ext}}
\newcommand {\lra}{\longrightarrow}
\newcommand {\A}{\mathcal{A}}
\newcommand {\Ra}{\mathcal{R}}
\newcommand {\la}{\langle}
\newcommand {\ra}{\rangle}

\begin{document}

\title{An algorithmic search for $\A$-annihilated classes in the Dyer-Lashof algebra and $H_*QS^0$ I. Closed form for low lengths and tables in low dimensions}

\author{Seyyed Mohammad Ali HasanZadeh\\
School of Mathematics. Sharif University of Technology. Tehran.\\
Islamic Republic of Iran\\
  e-mail\textup{: \texttt{m.ali.hasanzadeh1 at gmail.com}}\\
\and
Hadi Zare\\
School of Mathematics, Statistics, and Computer Science.\\
College of Science. University of Tehran.\\
Tehran \textup{14174}. Islamic Republic of Iran \\
  e-mail\textup{: \texttt{hadi.zare at ut.ac.ir}}
  }

\maketitle

\begin{abstract}
The aim of this work is to publicise some computational results involving tables which contain $\A$-annihilated monomials, excluding square classes, in the Dyer-Lashof algebra and $H_*QS^0$; our computations go up to dimension $1.1\times 10^7$ but the tables in this paper only announce results up to dimension $2^{17}=131072$ and full tables would be available upon request. The theoretical background for our computations is provided by work of Curtis \cite{Curtis} and Wellington \cite{Wellington} on the $A$-module structure of the Dyer-Lashof algebra as well as $H_*QS^0$. It seems to us that there is a workable algorithm to do these computations which we plan to announce in a future work, partly to avoid making this paper longer than it is. We hope to receive feedback from the experts on these computations and make our algorithm available as soon as we can. We hope that these tables provide a source for researchers in the field, as well as a pool of data to analyse the behaviour of these sequences, their distributions and other asymptotic behaviours. The problem of computing spherical classes in $H_*QS^0$ as well as the symmetric and non-symmetric hit problems have been our main motivations to pursue this project.\\
\textbf{Keywords:}Dyer-Lashof algebra, Steenrod algebra, hit problem.\\
\textbf{AMS subject classification:} 55S10, 55S12, 55P47
\end{abstract}


\section{Statement of results}
The $\Z/2$-homology of the space $QS^0\simeq\colim \Omega^iS^i$ which we denote by $H_*QS^0$ is given by
$$H_*QS^0\simeq\Z/2[Q^I[1]*[-2^{l(I)}]:I\textrm{ admissible}]$$
where $I=(i_1,\ldots,i_r)$ is a sequence of positive integers and it is called admissible if and only if $i_j\leqslant 2i_{j+1}$ for all $1\leqslant j\leqslant r-1$.  For any such $I$ we define its excess by $\ex(I)=i_1-(i_2+\cdots+i_r)$. The following is due to Curtis
\cite[Lemma 6.2, Theorem 6.3]{Curtis} (see also Wellington \cite[Theorem 5.6]{Wellington-thesis} as well as \cite{Wellington}).

\begin{thm}\label{ann}
Define $\phi:\mathbb{N}\lra\mathbb{N}\cup\{0\}$ by $\phi(n)=\min\{i:n_i=0\}$ for $n=\sum_{i=0}^\infty n_i2^i$ with $n_i\in\{0,1\}$.
For a generator $Q^I[1]$ of $H_*QS^0$, suppose $I=(i_1,\ldots,i_s)$ with $s>1$ is a sequence so that $\ex(I)<2^{\phi(i_1)}$ and $0\leqslant 2i_{j+1}-i_j<2^{\phi(i_{j+1})}$ for $1\leqslant j\leqslant s-1$. Then $Q^I[1]$ is $\A$-annihilated. If $I=(i)$ with $i<2^{\phi(i)}$, i.e. $i=2^t-1$ for some $t>0$, then $Q^i[1]$ is $\A$-annihilated. Here, $\A$ is the mod $2$ Steenrod algebra.
\end{thm}

Moreover, the result of Curtis and Wellington allows to determine all $\A$-annihilated classes in $H_*QX$ which are of the form $Q^Ix$ upon knowing $\A$-annihilated classes in $H_*X$ \cite[Theorem 1.2]{Zare-Glasgowjournal} with conditions similar to those of Theorem \ref{ann}. Consequently, the problem of determining $\A$-annihilated classes either in $H_*QS^0$ or $H_*QX$ requires one to determine all sequences $I=(i_1,\ldots,i_s)$ that would satisfy conditions of Theorem \ref{ann}. This latter problem is a very computations problem, and more of a combinatorial nature. Although, the work of \cite[Section 6]{Zare-Glasgowjournal} suggests a way to determine such sequences, however, it does not seem possible to use it to do machine based computations with it. In this paper, begin first in a sequel, making a better use of binary expansions, we offer an effective algorithm to compute such sequences. We shall provide some closed form for sequence $I$ of small length which satisfy conditions of Theorem \ref{ann}. We also do some preparatory observations which would allow us to build our algorithm in a forthcoming paper.\\

Our first main result determines all sequences of small length.

\begin{thm}\label{lowlength}
Suppose $I$ is an admissible sequence of positive excess satisfying conditions of Theorem \ref{ann}.\\
(i) If $l(I)=1$ then $I=(2^t-1)$ for some $t>0$. Moreover, any $I=(2^t-1)$ satisfies conditions of Theorem \ref{ann}.\\
(ii) For $l(I)=2$ there is no sequence which satisfies conditions of Theorem \ref{ann}.\\
(iii) If $l(I)=3$ then there exist positive integers $m$ and $n$ with $m<n-1$ so that
$$I=(2^{n+1}+2^m-1,2^n+2^m-1,2^n-1).$$
Moreover, any sequence $I$ of the above form satisfies conditions of Theorem \ref{ann}.\\
(iv) For $l(I)=4$ if $I$ satisfies conditions of Theorem \ref{ann} then $I$ has one of the following forms
$$\begin{array}{lll}
I=(2^{n+3}+2^{n-1}-1,2^{n+2}+2^{n-1}-1,2^{n+1}+2^{n-1}-1,2^n+2^{n-1}-1) & \textrm{with }n>2,\\
I=(2^{n+2}+2^{m+1}+2^{m'}-1,2^{n+1}+2^m+2^{m'}-1,2^n+2^m-1,2^n-1) & \textrm{with }n>m>m'>1.
\end{array}$$
Moreover, any sequence of the above forms satisfies conditions of Theorem \ref{ann}.\\
(v) For $l(I)=5$ if $I$ satisfies conditions of Theorem \ref{ann} then $I$ has one of the following forms
$$\begin{array}{lll}
I&=&(2^{n+4}+2^n-3,2^{n+3}+2^{n-1}-2,2^{n+2}+2^{n-1}-1,2^{n+1}+2^{n-1}-1,2^n+2^{n-1}-1) \\
 & &\textrm{with }n>2,\\
I&=&(2^{n+3}+2^{n+2}+2^{n-2}-1,2^{n+2}+2^{n+1}+2^{n-2}-1,2^{n+1}+2^n+2^{n-2}-1,\\
 & &2^n+2^{n-1}+2^{n-2}-1,2^n-1)\ \textrm{with }n>3,\\
I&=&(2^{n+3}+2^n+2^{m+1}+2^{m'}-1,2^{n+2}+2^{n-1}+2^{m}+2^{m'}-1,2^{n+1}+2^{n-2}+2^{m}-1,\\
 & &2^n+2^{n-2}-1,2^n-1)\ \textrm{with }n>m>m'>2,\\
I&=&(2^{n+3}+2^n+1,2^{n+2}+2^{n-1}+1,2^{n+1}+2^{n-2}+1,2^{n}+2^{n-2}-1,2^{n}-1)\\
 & &\textrm{with }n>2.
\end{array}$$
Moreover, any sequence of the above forms satisfies conditions of Theorem \ref{ann}.
\end{thm}

Let's note that part (i) and (ii) of the above Theorem are almost trivial, and known to Wellington. Parts (iii)-(v) are new. We provide proof for length $l(I)=1,2,3$. A proof of $l(I)\in\{4,5\}$ is delayed to be announced in a sequel as it depends on some additional results which we have avoided in order not to make this note longer than it is; meanwhile we will be happy to provide a proof upon request.\\
Next in line, is to search for an algorithm to do the computations and it is this line which we have followed. The following lists the results that we wish to announce in this paper, delayed a complete proof into another work.

\begin{thm}
(i) For dimension $d$ and length $s$ given, there is an algorithm which computes all sequences $I$ with $l(I)=s$ and $|I|=d$ which satisfy conditions of Theorem \ref{ann}.\\
(ii) Up to dimension $1.1\times 10^7$ there are less that $8\times 10^5$ sequences which satisfy conditions of Theorem \ref{ann}.
\end{thm}

Indeed, the observation (ii) is a computations result, but we believe that it is of its own interest.\\

\textbf{Acknowledgements.} The results of this project, and our upcoming paper, arise from an undergraduate project that the first named author has done under supervision of the second named author, while he has been an undergraduate student at the University of Tehran.

\section{Preparatory observations}
It is possible to use the conditions of Theorem \ref{ann} to put more restrictions on the sequences $I$. From computational point of view putting more restrictions on the set of all such sequences is the same as making the space in which we have to search for such sequences smaller. \\

We have the following.

\begin{lmm}\label{oddentries}
(i) For $r>1$,  if $I=(i_1,\ldots,i_r)$ is an admissible sequence with $\ex(I)>0$ then $I$ is strictly decreasing, i.e. $i_j>i_{j+1}$ for all $1\leqslant j<r$.\\
(ii) Suppose $I=(i_1,\ldots,i_s)$ is a sequence with $\ex(I)>0$ which satisfies conditions of Theorem \ref{ann}. Then, all entries of $I$ are odd.
\end{lmm}

\begin{proof}
(i) This easily follows by induction from the admissibility and positivity of excess.\\
(ii)It is straightforward to see that if $s=1$ then $I=(2^t-1)$ for some $t>0$.\\
Suppose $s>0$ and $I=(i_1,\ldots,i_r)$ is an admissible sequence with $\ex(I)>0$ which satisfies condition of Theorem \ref{ann}. First, note that if $i_1$ is even then $\phi(i_1)=0$ which together with condition $0<\ex(I)<2^{\phi(i_1)}$ shows that $0<\ex(I)<1$ which is a contradiction. Hence, $i_1$ must be odd. For $j>1$ suppose $i_j$ is even, hence $\phi(i_j)=0$. The condition $0\leqslant 2i_j-i_{j-1}< 2^{\phi(i_j)}=1$ implies that $i_{j-1}=2i_{j}$. By iterating this process we see that $i_1$ must be even which is a contradiction. Hence, $I$ consists of only odd entries.
\end{proof}

The following is now evident.

\begin{crl}\label{parity}
If $I$ satisfies conditions of Theorem \ref{ann} then $l(I)$ and $|I|$ have the same parity.
\end{crl}

Next, note that for a sequence $I=(i_1,\ldots,i_r)$ we have its dimension $|I|=i_1+\cdots+i_r$. We have the following.

\begin{lmm}\label{reduction1}
Suppose $I=(i_1,\ldots,i_r)$ is an admissible sequence of positive excess. Then $I$ satisfies conditions of Theorem \ref{ann} if and only if for $(i_0,I):=(i_0,i_1,\ldots,i_r)$ we have $0<2i_{j+1}-i_j<2^{\phi(i_{j+1})}$ for all $j\in\{0,\ldots,r-1\}$ where $i_0=|I|$.
\end{lmm}

\begin{proof}
The inequality $0<\ex(I)<2^{\phi(i_1)}$ is the same as $0<2i_1-i_0<2^{\phi(i_1)}$ if we replace $\ex(I)=2i_1-i_0$. The result now follows.
\end{proof}

The above observations are important as they replace two conditions of Theorem \ref{ann} only with one condition, so that search for sequences satisfying condition of Lemma \ref{reduction1} would also provide us with all sequences which satisfy conditions of Theorem \ref{ann} and the refinement would be $i_0=|I|=i_1+\cdots+i_r$.

\section{Towards constructing an algorithm: auxiliary functions}
The main ingredient of Theorem \ref{ann} is the function $\phi$ which allows to express the result in a combinatorial manner. We are therefore interested in studying properties of this function. To begin with, lets recall the following which is implicit in Curtis's work and follows by looking at the binary expansion of numbers.

\begin{lmm}\label{increasing}
Suppose $I=(i_1,\ldots,i_r)$ is an admissible sequence with $\ex(I)>0$ satisfying conditions of Theorem \ref{ann}. Then,
$$\phi(i_1)\leqslant\cdots\leqslant\phi(i_r).$$
\end{lmm}

We also define $\psi:\N\to\N$ by
$$\psi(n)=\max\{i:n_i=1\}+1=\min\{i:\forall j\geqslant i,\ n_j=0\}.$$
In fact, writing $n\in\N$ in binary expansion $\cdots n_i\cdots n_1n_0$, that is $n=\sum_{i=0}^{+\infty} n_i2^i$ with $n_i\in\{0,1\}$, the function $\psi$ assigns to $n$ the `length' of its binary expansion.\\

We call $n$ a spike if $n=2^t-1$ for some $t>0$. The following is immediate.

\begin{lmm}\label{spike1}
For $n\in\N$, $\phi(n)\leqslant \psi(n)$. Moreover, $\phi(n)=\psi(n)$ if and only if $n=2^{t}-1$ with $t=\psi(n)-1$.
\end{lmm}

We also have the following characterisation of non-spike integers.

\begin{lmm}\label{nonspike1}
(i) Let $n$ be a non-spike positive integer, that is $n\neq 2^t-1$ for all $t$. Then, there exists a natural number $N_n>1$ so that $n=2^{\phi(n)}N_n+2^{\phi(n)}-1$. In particular, $N_n$ is an even number.\\
(ii) If $n$ is a non-spike then there exits a positive integer $2^{\phi(n)}<B(n)$ such that
$$n=2^{\psi(n)-1}+B(n)+2^{\phi(n)}-1.$$
\end{lmm}

\begin{proof}
(i) Note that for any positive integer $k$ we have $2^k-1=\sum_{j=0}^{k-1} 2^{j}$. This together with definition of $\phi(n)$ shows that for $n=\sum_{i=0}^{+\infty}n_i2^i$ we have
$$N_n=(n-2^{\phi(n)}-1)/2^{\phi(n)}=\sum_{i=\phi(n)}^{+\infty}n_i2^{i-\phi(n)}.$$
If $N_n$ is an odd number then $n_{{\phi(n)}}=1$ which contradicts definition of $\phi(n)$. In particular, $N_n$ cannot be $1$.\\
(ii) by definition of $\psi$ we have $n_{\psi(n)-1}=1$ and $n_i=0$ for all $i\geqslant\psi(n)$.
For $n=\sum_{i=0}^{+\infty}n_i2^i$ written in binary form, $B(n)=\sum_{i=\phi(n)+1}^{\psi(n)-2}n_i2^i$ is the required value.
\end{proof}

Sometimes we refer to $B(n)$ is the (undetermined-) block of $n$.\\

The functions $\phi$ and $\psi$ provide some useful upper and lower bounds non-spike numbers. Before proceeding further, let's recall that for all $i>0$ we have
$$2^i-1=\sum_{j=0}^{i-1}2^j.$$
Most of what we say below are consequences of this equality. For instance, we have the following.

\begin{lmm}\label{bounds}
(i) For $n\in\N$, $n_{\psi(n)}=0$ and $n_{\psi(n)-1}=1$.\\
(ii) For any non-spike positive integer $n\in\N$ we have
$$2^{\psi(n)-1}+2^{\phi(n)}-1\leqslant n\leqslant 2^{\psi(n)}-1-2^{\phi(n)}<2^{\psi(n)}-1.$$
(iii) For any non-spike positive integer $n$ we have $\phi(n)<\psi(n)-1$.
\end{lmm}

\begin{proof}
(i) This follows from the definition of $\psi$.\\
(ii) Using the equality $2^i-1=\sum_{j=0}^{i-1}2^j$ the inequality is immediate.\\
(iii) This is also immediate from the definitions.
\end{proof}

The following provides an application of the inequalities of Lemma \ref{bounds}.

\begin{lmm}\label{spikeelimination}
Suppose $I=(i_1,\ldots,r)$ is an admissible sequence with $\ex(I)>0$ which satisfies conditions of Theorem \ref{ann}. If $i_j=2^n-1$ for some $n>0$ then $j=r$.
\end{lmm}

\begin{proof}
By Lemma \ref{oddentries}(i) the sequence $I$ is strictly decreasing. Suppose $j<r$ and $i_j=2^n-1$ for some $n>0$. Then $i_{j+1}<2^n-1$. We compare $\psi(i_{j+1})$ to $n$ and proceed to show that any possible choice for $\psi(i_{j+1})$ leads to a contradiction. We consider the following cases.\\
\textbf{Case of $\psi(i_{j+1})\geqslant n+1$.} In this case, $\psi(i_{j+1})-1\geqslant n$. By Lemma \ref{bounds}(ii) we have
$$i_{j+1}\geqslant 2^{\psi(i_{j+1})-1}+2^{\phi(i_{j+1})}-1>2^{\psi(i_{j+1})-1}-1\geqslant 2^n-1=i_j$$
which is a contradiction.\\
\textbf{Case of $\psi(i_{j+1})\leqslant n-1$.} Using the inequalities of Lemma \ref{bounds}(ii) we see that
$$i_{j+1}\leqslant 2^{\psi(i_{j+1})}-1\leqslant 2^{n-1}-1$$
which together with $i_j=2^n-1$ and the admissibility condition implies that
$$2i_{j+1}\leqslant 2^n-2=i_j\leqslant 2i_{j+1}$$
Hence, $i_j=2i_{j+1}$ which contradicts the fact that $i_j$ is odd.\\
\textbf{Case of $\psi(i_{j+1})=n$.} If $\phi(i_{j+1})=\psi(i_{j+1)}=n$ then by Lemma \ref{spike1} $i_{j+1}=2^{n-1}-1$. The admissibility condition reads as
$$i_j=2^n-1\leqslant 2i_{j+1}=2^n-2$$
which is a contradiction. Hence, $\phi(i_{j+1})<\psi(i_{j+1})$, i.e. $i_{j+1}$ is not a spike. Hence, by Lemma \ref{bounds}(ii) we have
$$i_{j+1}\geqslant 2^{\psi(i_{j+1})}+2^{\phi(i_{j+1})}-1=2^{n-1}+2^{\phi(i_{j+1})}-1.$$
Noting that $i_j=2^n-1$, and multiplying both sides of this inequality by $2$ we have
$$2i_{j+1}\geqslant 2^{n}+2^{\phi(i_{j+1})+1}-2$$
which implies that $2i_{j+1}-i_j\geqslant 2^{\phi(i_{j+1})+1}-1$. On the other hand, as $I$ satisfies conditions of Theorem \ref{ann}, $2i_{j+1}-i_j\leqslant 2^{\phi(i_j+1)}$ which together with Lemma \ref{oddentries}(ii) implies that $2i_{j+1}-i_j<2^{\phi(i_j+1)}$. These together imply that
$$2^{\phi(i_{j+1})+1}-1<2^{\phi(i_j+1)}$$
which is a contradiction.\\
This completes the proof.
\end{proof}

A very immediate corollary of the proof is the following.

\begin{crl}\label{strictphi}
If $I=(i_1,\ldots,i_r)$ is an admissible sequence of positive excess, satisfying conditions of Theorem \ref{ann}, with $i_r=2^{\phi(i_r)}-1$ then $\phi(i_{r-1})<\phi(i_r)$.
\end{crl}

We can prove a more practical criterion of the sequence of $\psi$'s. We have the following.

\begin{thm}\label{psisequence}
Suppose $I=(i_1,\ldots,i_r)$ is an admissible sequence of positive excess, satisfying conditions of Lemma \ref{reduction1}. Then for all $j\in\{2,\ldots,r\}$ we have $\psi(i_j)=\psi(i_{j-1})-1$.
\end{thm}

\begin{proof}
We eliminate the cases with $\psi(i_j)\leqslant \psi(i_{j-1})-2$ and $\psi(i_j)\geqslant \psi(i_{j-1})$ as follows.\\
\textbf{Case of $\psi(i_j)\leqslant \psi(i_{j-1})-2$.} We have
$$i_j\leqslant 2^{\psi(i_j)}-1\leqslant 2^{\psi(i_{j-1})-2}-1$$
which implies that $2i_j\leqslant 2^{\psi(i_{j-1})-1}-2$. Now, for any $j\in\{2,\ldots,r\}$, since $j-1\neq r$ then $i_{j-1}$ is not a spike. Hence, $i_{j-1}\geqslant 2^{\psi(i_{j-1})}+2^{\phi(i_{j-1})}-1$. These together imply that
$$2i_j-i_{j-1}<-2^{\phi(i_{j-1})}<0$$
which is a contradiction.\\
\textbf{Case of $\psi(i_j)=\psi(i_{j-1})$.} First, suppose $i_j$ is not a spike. Then,
$$i_j\geqslant 2^{\psi(i_j)}+2^{\phi(i_j)}-1=2^{\psi(i_{j-1})}+2^{\phi(i_j)}-1.$$
Moreover, since by Lemma \ref{spikeelimination} $i_{j-1}$ is not a spike then $i_{j-1}\leqslant 2^{\psi(i_{j-1})}-1$. These together imply that $2i_j-i_{j-1}>2^{\phi(i_j)}$ which is a contradiction. Hence, $i_j$ has to be a spike and $j=r$ by Lemma \ref{spikeelimination}, that is $i_r=2^{\phi(i_r)}-1$. In this case, as $i_{j-1}$ is not a spike, we have
$$i_{j-1}\leq 2^{\psi(i_{j-1})}-1-2^{\phi(i_{j-1})}.$$
These together imply that
$$2i_r-i_{r-1}\geqslant 2^{\phi(i_r)}+2^{\phi(i_{r-1})}-1>2^{\phi(i_r)}$$
which contradicts conditions of Theorem \ref{ann}.\\
The above computations leave us with $\psi(i_j)=\psi(i_{j-1})-1$ as the only choice. This completes the proof.
\end{proof}

\section{Closed forms for low lengths}
The aim of this section is to prove Theorem \ref{lowlength}.

\subsection{Case of $l(I)=1,2$}

\begin{lmm}
Suppose $I=(i_1,\ldots,i_s)$ is an admissible sequence.\\
(i) For $l(I)=1$, $I$ satisfies conditions of Theorem \ref{ann} if and only if $I=(2^t-1)$ for some $t>0$.\\
(ii) If $l(I)=2$ then $I$ does not satisfy conditions of Theorem \ref{ann}.
\end{lmm}

\begin{proof}
(i) It is clear from binary expansion.\\
(ii) By Lemma \ref{nonspike1} we have $i=2^{\phi(i)}-1+N_i2^{\phi(i)}$
for some positive integer $N_i\neq 1$. If $I=(i_1,i_2)$ is a sequence satisfying conditions of Theorem \ref{ann} then $\phi(i_1)\leqslant\phi(i_2)$. The conditions
$$0\leqslant 2i_2-i_1<2^{\phi(i_2)},\ i_1-i_2<2^{\phi(i_1)}$$
imply that
$$i_2=2^{\phi(i_2)}-1+N_{i_2}2^{\phi(i_2)}<2^{\phi(i_2)+1}\Rightarrow N_{i_2}\leqslant 1.$$
If $i_2$ is not a spike then this is a contradiction by Lemma \ref{nonspike1}. The only remaining case is that $i_2$ is a spike which corresponds to the case $N_{i_2}=0$ and $i_2=2^{\phi(i_2)}-1$. Note that $i_1=2^{\phi(i_1)}-1+N_{i_1}2^{\phi(i_1)}$. Using the conditions of the theorem simultaneously, yields
$$\begin{array}{lllllllllll}
i_1-i_2  &<& 2^{\phi(i_1)}  &\Rightarrow& N_{i_1}&<&2^{\phi(i_2)-\phi(i_1)},\\
2i_2-i_1 &<& 2^{\phi(i_2)}  &\Rightarrow& N_{i_1}&>&2^{\phi(i_2)-\phi(i_1)}-1.
\end{array}$$
This gives the desired contradiction.
\end{proof}

\subsection{Case of $l(I)=3$}

\begin{lmm}\label{l(I)=3-1}
If $I=(i_1,i_2,i_3)$ is an admissible sequence which satisfies conditions of Theorem \ref{ann} then $i_3=2^n-1$ for some positive integer $n$.
\end{lmm}

\begin{proof}
Suppose $I=(i_1,i_2,i_3)$ is a sequence which satisfies conditions of Theorem \ref{ann}. For brevity, we write $\phi_j=\phi(i_j)$ and $N_j=N_{i_j}$. The conditions of Theorem \ref{ann} for $I$ reads as
$$(1)\ i_1-(i_2+i_3)<2^{\phi_1},\ (2)\ 2i_2-i_1<2^{\phi_2},\ (3)\ 2i_3-i_2<2^{\phi_3}.$$
We proceed by some eliminations.\\
\textbf{Case $i_3$ non-spike and $\phi_2<\phi_3$}. If $i_3$ is not a spike then $i_3=2^{\phi_3}N_3+2^{\phi_3}-1$ for some $N_3>1$.
By adding the three inequalities above, we obtain
$$i_3<2^{\phi_1}+2^{\phi_2}+2^{\phi_3}.$$
By plugging in the value of $i_3$ as above and using Lemma \ref{increasing} yields
$$2^{\phi_3}N_3<2^{\phi_1}+2^{\phi_2}+1<2^{\phi_2+1}+1.$$
Since $\phi_2<\phi_3$ hence $\phi_2+1\leqslant\phi_3$. This implies that $N_3\leqslant 1$ which is a contradiction by Lemma \ref{nonspike1}.\\
\textbf{Case $i_3$ non-spike and $\phi_2=\phi_3>1$}. Suppose $\phi_2=\phi_3=k>1$ for some $k$. By adding the inequalities $(1)$ and $(2)$ as well as noting that by $\phi_1\leqslant k$ by Lemma \ref{increasing} we obtain
$$i_2-i_3<2^{k+1}\Rightarrow N_2-N_3<2^k.$$
Moreover, the inequality $(3)$ implies that $2^kN_3\leqslant N_2$. By adding these two inequalities, we obtain $N_3\leq 2^k/(2^k-1)<2$ (for $k>1$) which contradicts Lemma \ref{nonspike1}. \\
\textbf{Case $i_3$ non-spike and $\phi_2=\phi_3=1$}. By Lemma \ref{oddentries} all entries of $I$ are odd. Hence, $\phi_1>0$. This together with Lemma \ref{increasing} implies that $\phi_1=\phi_2=\phi_3=1$. The conditions of Theorem \ref{ann} imply that
$$i_1=2i_2-1,i_2=2i_3-1.$$
By computations of the previous case, $N_3\leqslant 2^{k}/(2^k-1)$ where $k=\phi_3=\phi_2$. Hence, $N_3\leqslant 2$ which together with Lemma \ref{nonspike1} forces $N_3=2$. Hence, $i_3=1+2\times 2=5$. Consequently, $I=(17,9,5)$. However, this does not satisfy condition $(1)$ of Theorem \ref{ann} as $i_1-(i_2+i_3)=3\not<2^1$.
\end{proof}

Next, we compute $i_2$ and $i_3$.

\begin{lmm}\label{l(I)=3-2}
Suppose $I=(i_1,i_2,i_3)$ is an admissible sequence satisfying conditions of Theorem \ref{ann}. Then, for some $m\leqslant n-1$
$$I=(2^{n+1}+2^m-1,2^n+2^m-1,2^n-1).$$
\end{lmm}

\begin{proof}
We shall write $\psi_j=\psi(i_j)$ and $\phi_j=\phi(i_j)$. By Lemma \ref{l(I)=3-1} $i_3=2^n-1$ for some $n>0$. It follows that $\phi_3=\psi_3=n$. By Lemma \ref{psisequence} it follows that $\psi_2=n+1$ and $\psi_1=n+2$. Moreover, since $i_3$ is a spike then by Lemmata \ref{strictphi} and \ref{increasing} we have $\phi_1\leqslant\phi_2<\phi_3=n$. Applying Lemma \ref{bounds}(iii) we see that
$$i_1=2^{n+1}+B_1+2^{\phi_1}-1,\ i_2=2^n+B_2+2^{\phi_2}-1$$
where $B_1=\sum_{k_1=\phi_1+1}^n\alpha_{k_1}2^{k_1}$ and $B_2=\sum_{k_2=\phi_2+1}^{n-1}\alpha_{k_2}2^{k_2}$. We claim that $B_1=B_2=0$. We proceed as follows.\\
First, suppose $B_2>0$. The conditions $0<2i_2-i_1<2^{\phi_2}$ and $0<\ex(I)<2^{\phi_1}$ imply that
$$\begin{array}{lll}
2^{n+1}+2B_2+2^{\phi_2+1}-2-(2^{n+1}+B_1+2^{\phi_1}-1)<2^{\phi_2} &\Rightarrow& 2B_2-B_1+2^{\phi_2}-1<2^{\phi_1},\\
0<2^{n+1}+B_1+2^{\phi_1}-1-(2^n+B_2+2^{\phi_2}-1+2^n-1)<2^{\phi_1}&\Rightarrow& B_1-B_2-1<2^{\phi_2}.
\end{array}$$
By adding the resulting inequalities we have $B_2-2<2^{\phi_1}\leqslant 2^{\phi_2}$ which from $B_2$ being even we deduce that $B_2\leqslant 2^{\phi_2}$. But, this is a contradiction as if $B_2>0$ then from its expression $B_2>2^{\phi_2}$. Hence, $B_2=0$ and consequently $i_2=2^n+2^{\phi_2}-1$. Therefore, for some $m\leqslant n-1$ we have $\phi_2=m$ and
$$I=(2^{n+1}+B_1+2^{\phi_1}-1,2^n+2^m-1,2^n-1)$$
where $\phi_1\leqslant m$. Next, we turn to $B_1$. \\
First, suppose $B_1=0$. In this case if $\phi_1<m$ then $\ex(I)<0$ which is a contradiction. Hence, in the case of $B_1=0$ we have $\phi_1=m$ and
$$I=(2^{n+1}+2^{m}-1,2^n+2^m-1,2^n-1).$$
We show that the assumption that $B_1>0$ leads to a contradiction. The condition $0<2i_2-i_1<2^{\phi_2}=2^m$ implies that
$$2^m\leqslant B_1+2^{\phi_1}<2^{m+1}.$$
Consider the binary expansion of $B_1$ as $B_1=\sum_{k_1=\phi_1+1}^{\psi_1-2}\alpha_{k_1}2^{k_1}$ bearing in mind that $\psi_1-2=n$. The above inequalities shows that
$$\alpha_{k_1}=0\ \textrm{for all }k_1\geqslant m+1.$$
Hence, $B_1=\sum_{k_1=\phi_1+1}^{m}\alpha_{k_1}2^{k_1}$. If $\phi_1<m$ then the condition $B_1+2^{\phi_1}\geqslant 2^m$ reads as $\sum_{k_1=\phi_1+1}^{m}\alpha_{k_1}2^{k_1}+2^{\phi_1}\geqslant 2^m$ which is impossible if $\alpha_m=0$. Hence, $\alpha_m=1$. Therefore,
$$i_1=2^{n+1}+2^m+\sum_{k_1=\phi_1+1}^{m-1}\alpha_{k_1}2^{k_1}+2^{\phi_1}-1.$$
The condition $\ex(I)<2^{\phi_1}$ reads as
$$\sum_{k_1=\phi_1+1}^{m-1}\alpha_{k_1}2^{k_1}+2^{\phi_1}+1<2^{\phi_1}$$
which is a contradiction if $\alpha_{k_1}=1$ for some $k_1\in\{\phi_1+1,\ldots,m-1\}$. Therefore, $B_1=2^m$ is the only choice. However, if $\phi_1<m$ then $i_1=2^{n+1}+2^m+2^{\phi_1}-1$ and the excess condition $\ex(I)<2^{\phi_1}$ implies that $2^{\phi_1}+1<2^{\phi_1}$ which is a contradiction. Hence, $B_1=0$. This completes the proof.
\end{proof}

As an example where the boundary value $m=n-1$ could be attained, consider $I=(19,11,7)$ where $11=2^3+2^{3-1}-1$.

The process of eliminating $B_1$ in the above proof is more of a intuitive nature and would be immediate if one has the experience and passion for working with binary expansions.

\section{Qualitative observations}
Since we don't have a closed form for the sequences $I$ satisfying conditions of Theorem \ref{ann}, besides the fact that we have an algorithm to compute all such sequences of a given dimension and length, we wish to have some information about the distribution of such sequences. Notice that given a length $s$ then the initial search for sequences of length $l(I)=s$, with no restriction on the dimension, is to be done in $\N^{\times s}$ which after putting a bound on the dimension becomes a finite space. That is searching for all sequences $I$ of dimension at most $|I|\leqslant d$ should be within a finite number of points in the union of lattices $\cup_{s=1}^{+\infty}\N^{\times s}$. For the purpose of using our results in the studying hit problem it is more convenient to count the number of such sequences whose dimensions are no more than a given dimension $r$. For this purpose, Lemma \ref{lowerbound} below is a helpful observation.

\begin{lmm}\label{lowerbound}
Suppose $I$ is an admissible sequence of positive excess (hence strictly decreasing) having only odd entries. If $|I|=s$ and $|I|>1$ then $d>2^s$.
\end{lmm}

We note that one implication of Lemma \label{upperbound} is that it is useful to concentrate on intervals such as $[2^\alpha,2^{\alpha+1}]$ of integer numbers which is quite desirable when we work at the prime $2$. Indeed, later on, we shall see conditions of Theorem \ref{ann} to improve this bound much further. The possibility of such an improvement could be noted by looking at our tables. For instance, we see from our tables that the maximum length for classes up to dimension $16348$ is $5$ which is much less that $14=\log_2(16348)$. We will achieve such bounds later on.

\appendix

\section{Tables with dimension range}
This section presents some of the sequences we have found, not based on their length, but depending on their dimensions falling into intervals $[2^n,2^{n-1}]$ of integer numbers. The left column denotes dimension $d=a_1+\cdots+d_r$ of a sequence $(a_1,\ldots,a_r)$, placed at the middle column, whose length $r$ is shown on the right column.

\begin{longtable}{|c|l|c|}
	KILLED & LINE!!!!\kill
        \hline
		\textrm{dimension}	&$(a_1,a_2, ... , a_r)$&\textrm{length}\\
			\hline
			1&  1&1\\
			3&  3&1\\
			7&  7&1\\
			15&  15&1\\
			17&  9  ,   5  ,   3&3\\
			31&  31&1\\
			33&  17  ,   9  ,   7&3\\
\end{longtable}

\end{document}